\documentclass[]{amsart}
\usepackage{amscd,amsthm,amssymb,amsfonts,amsmath,euscript}

\usepackage{graphicx}
\theoremstyle{plain}
\newtheorem{thm}{Theorem}[section]
\newtheorem{lemma}[thm]{Lemma}
\newtheorem{prop}[thm]{Proposition}

\theoremstyle{definition}

\theoremstyle{remark}
\newtheorem{remark}[thm]{Remark}

\newcommand{\nc}{\newcommand}

\def\makeop#1{\expandafter\def\csname#1\endcsname
  {\mathop{\rm #1}\nolimits}\ignorespaces}
\makeop{Hom}   \makeop{End}   \makeop{Aut}   \makeop{Isom}  \makeop{Pic}
\makeop{Gal}   \makeop{ord}   \makeop{Char}  \makeop{Div}   \makeop{Lie}
\makeop{PGL}   \makeop{Corr}  \makeop{PSL}   \makeop{sgn}   \makeop{Spf}
\makeop{Spec}  \makeop{Tr}    \makeop{Nr}    \makeop{Fr}    \makeop{disc}
\makeop{Proj}  \makeop{supp}  \makeop{ker}   \makeop{im}    \makeop{dom}
\makeop{coker} \makeop{Stab}  \makeop{SO}    \makeop{SL}    \makeop{SL}
\makeop{Cl}    \makeop{cond}  \makeop{Br}    \makeop{inv}   \makeop{rank}
\makeop{id}    \makeop{Fil}   \makeop{Frac}  \makeop{GL}    \makeop{SU}
\makeop{Trd}   \makeop{Sp}    \makeop{Tr}    \makeop{Trd}   \makeop{diag}
\makeop{Res}   \makeop{ind}   \makeop{depth} \makeop{Tr}    \makeop{st}
\makeop{Ad}    \makeop{Int}   \makeop{tr}    \makeop{Sym}   \makeop{can}
\makeop{length}\makeop{SO}    \makeop{torsion} \makeop{GSp} \makeop{Ker}
\makeop{Adm}
\def\makebb#1{\expandafter\def
  \csname bb#1\endcsname{{\mathbb{#1}}}\ignorespaces}
\def\makebf#1{\expandafter\def\csname bf#1\endcsname{{\bf
      #1}}\ignorespaces}
\def\makegr#1{\expandafter\def
  \csname gr#1\endcsname{{\mathfrak{#1}}}\ignorespaces}
\def\makescr#1{\expandafter\def
  \csname scr#1\endcsname{{\EuScript{#1}}}\ignorespaces}
\def\makecal#1{\expandafter\def\csname cal#1\endcsname{{\mathcal
      #1}}\ignorespaces}

\def\doLetters#1{#1A #1B #1C #1D #1E #1F #1G #1H #1I #1J #1K #1L #1M
                 #1N #1O #1P #1Q #1R #1S #1T #1U #1V #1W #1X #1Y #1Z}
\def\doletters#1{#1a #1b #1c #1d #1e #1f #1g #1h #1i #1j #1k #1l #1m
                 #1n #1o #1p #1q #1r #1s #1t #1u #1v #1w #1x #1y #1z}
\doLetters\makebb   \doLetters\makecal  \doLetters\makebf
\doLetters\makescr
\doletters\makebf   \doLetters\makegr   \doletters\makegr
     \def\qed{\qedmark\medbreak}%
\def\qedmark{{\enspace\vrule height 6pt width 5pt depth 1.5pt}}%

\normalsize

\makeop{Bl}

\newcommand{\Z}{\mathbb Z}

\nc{\embed}{\hookrightarrow}

\nc{\ol}{\overline}
\nc{\wt}{\widetilde}
\nc{\opp}{\mathrm{opp}}
\def\ul{\underline}

\makeop{Ram}
\makeop{Rep}
\makeop{lcm}

\begin{document}
\renewcommand{\thefootnote}{\fnsymbol{footnote}}
\setcounter{footnote}{-1}
\numberwithin{equation}{section}


\title[]
{Monomial, Gorenstein and Bass Orders}
\author{Tse-Chung Yang and Chia-Fu Yu}

\address{
(Yang) Department of Mathematics, National Taiwan University\\
Astronomy Mathematics Building \\
No. 1, Roosevelt Rd. Sec. 4 \\
Taipei, Taiwan, 10617}
\email{d97221007@ntu.edu.tw}

\address{
(Yu) Institute of Mathematics, Academia Sinica and NCTS (Taipei Office)\\
Astronomy Mathematics Building \\
No. 1, Roosevelt Rd. Sec. 4 \\
Taipei, Taiwan, 10617}
\email{chiafu@math.sinica.edu.tw}

\date{\today}

\def\c{{\rm c}}
\def\i{{\rm i}}
\def\Mat{{\rm Mat}}

\begin{abstract}
In this article we study a class of orders called {\it monomial
orders} in a central simple algebra over a non-Archimedean local
field. Monomial orders are easily represented and they may be also
viewed as a direct generalization of Eichler orders
in quaternion algebras. A criterion for monomial orders to be
Gorenstein or to be Bass is given.
It is shown that a monomial order is Bass if and only if it is either a
hereditary or an Eichler order of period two.
\end{abstract}

\maketitle


\section{Introduction}
\label{sec:01}
In the integral theory for central simple algebras over
non-Archimedean local fields, we have the following important classes of
orders: maximal orders, hereditary orders, Bass orders and Gorenstein
orders. For definitions and basic properties of Bass and Gorenstein
orders, see \cite[Section 37]{curtis-reiner:1},
\cite[Chapter IX] {roggenkamp:2} or the original paper by Drozd,
Kirichenko and Roiter \cite{DKR:1967} (also see Section~\ref{sec:33}).
Many theories and relations
are investigated by many authors
for understanding these classes of orders, as well as
their module structures.
It is well-known that they form
the following proper inclusions:
\begin{center}
  (maximal orders) $\subset$ (hereditary orders)
                   $\subset$ (Bass orders) $\subset$ (Gorenstein orders)
\end{center}

Let $k$ be a non-Archimedean local field with ring of integers
$\calO$. In this article we introduce a class of $\calO$-orders in a
central simple algebra $A$ over $k$, which we call {\it monomial
orders}. To define them, let us write $A=\Mat_n(D)$,
where $D$ is a central division
algebra over $k$. Let $\calO_D$ be the unique maximal order in $D$
and $\grP$ be the unique maximal ideal of $\calO_D$. For any
matrix $\ul m$ in $\Mat_n(\Z)$, put
\begin{equation}
  \label{eq:1.1}
\Mat_n(\calO_D,\ul m):=\{(a_{ij})\in \Mat_n(D)\, ; \,
a_{ij}\in \grP^{m_{ij}}\ \forall\, 1\le i,j\le n \, \}.
\end{equation}
We call $\ul m$ the {\it level} of the $\calO_D$-module $\Mat_n(\calO_D,\ul
m)$. In order that $\Mat_n(\calO_D, \ul m)$ forms an order the matrix
$\ul m$ must satisfy the {\it order condition}: $m_{ii}=0$ for all
$1\le i \le n$ and $m_{ik}\le m_{ij}+m_{jk}$
for all $1\le i,j,k\le n$.
An $\calO$-order $R$ in $A$ that is isomorphic to
$\Mat_n(\calO_D, \ul m)$ for some $\ul m$ is called a {\it monomial
order}. 
We call $\Mat_n(\calO_D, \ul m)$ the standard monomial order of level $\ul
m$ in $\Mat_n(D)$.
Note that the level is not an intrinsic invariant
of a monomial order; a standard monomial order $\Mat_n(\calO,\ul m)$
may be conjugate to another $\Mat_n(\calO_D, \ul m')$ with different
level $\ul m'$.
The class of monomial orders may be viewed as a
direct generalization of Eichler orders in quaternion
algebras to central simple algebras.
We shall discuss properties of monomial orders and therefore
we shall consider only the standard ones. Clearly one has the
following proper inclusions:
\begin{center}
  (maximal orders) $\subset$ (hereditary orders)
                   $\subset$ (monomial orders)
\end{center}

The main content of this paper is the determination of
which monomial orders that are Gorenstein or Bass.
One basic problem studied by
Janusz~\cite{janusz:tpo1979} and extended by Hijikata and
Nishida~\cite{hijikata-nishida:tensor} is to determine
which tensor product order $R_1\otimes_{\calO} R_2$ of two orders
$R_1$ and $R_2$ is hereditary.
Naturally one may consider the same question where
the adjective ``hereditary'' is replaced by ``Bass'' or ``Gorenstein''.
Our investigation starts by a key observation due
to J.-K. Yu \cite{jk:cyclic} which asserts that
any hereditary order or the tensor
product of any two hereditary orders, even of two monomial orders,
locally in the etale topology looks like a monomial order.
That is, for such orders $R$,
there exists a finite etale extension $\calO'$ of $\calO$ such that
the base change $R\otimes_\calO \calO'$ becomes a monomial order.
Therefore, monomial orders are, up to a suitable base change,
stable under the tensor product.
Studying monomial orders and their properties may provide a
tool or interesting examples that emerge previous extensive studies
in the integral theory of central simple algebras over local fields.

Recall an $\calO$-order
$R$ in a semi-simple and separable $k$-algebra
is said to be {\it Gorenstein} if every short exact
sequence of right (or equivalently, left) $R$-lattices
\[ 0\to R\to M\to N\to 0 \]
splits; see \cite[p. 776 and Prop. 37.8, p. 778]{curtis-reiner:1}. We
prove the following result.

\begin{thm}\label{1.1}
  Let $R=\Mat_n(\calO_D, \ul m)$ be a monomial order of level $\ul
  m=(m_{ij})$ in $\Mat_n(D)$. Then the order $R$ is
  Gorenstein if and only if for each $1\le i\le n$ there exists an
  integer $c(i)$ such that the column vector
  $[-m_{i1}+c(i),\dots ,-m_{in}+c(i)]^t\in \Z^n$ is
  equal to a column of $\ul m$.
\end{thm}

We say a (standard) monomial order $\Mat_n(\calO_D, \ul m)$ of level
$\ul m$ {\it upper triangular} if $m_{ij}=0$ for all $1\le i<j\le n$.
We call this order {\it an Eichler order} if it is upper triangular
and for any $i>j$ the integer $m_{ij}$ is either zero
or equals a fixed positive integer.
The following result gives a criterion for an upper
triangular monomial order to be Gorenstein.

\begin{thm}\label{1.2}
If $R=\Mat_n(\calO_D,\ul m)$ is an upper triangular monomial order,
then $R$ is Gorenstein if and only if $R$ is an Eichler order.
\end{thm}


The following result characterizes Bass orders in the class of
monomial orders.

\begin{thm}\label{1.3}
  Let $R=\Mat_n(\calO_D, \ul m)$ be as in Theorem~\ref{1.1}.
  Then $R$ is a
  Bass order if and only if either $R$ is a hereditary order,
  or $R$ is an Eichler order of period two,
  i.e. there exist positive integers
  $k_1, k_2, a$ with $k_1+k_2=n$ so that $R \simeq \Mat_n(\calO_D, \ul
  m')$ with
\[ \ul m'= \left[
\begin{array}{cc}
 {\bf 0}^{k_1}_{k_1} &  {\bf 0}^{k_1}_{k_2} \\
a\, {\bf 1}^{k_2}_{k_1} & {\bf 0}^{k_2}_{k_2}
\end{array} \right],\]
where ${\bf 1}^{r}_{s}$ (resp. ${\bf 0}^{r}_{s}$) denotes the $r\times s$ matrix with every
entry equal to $1$ (resp. $0$).
\end{thm}

\begin{remark}
  Although the setting for the present article is over the ring
  $\calO$ of integers
  in a non-Archimedean local field, all results are also valid if one
  replaces the base ring $\calO$ by any complete discrete valuation ring.
\end{remark}

The paper is organized as follows. Section~\ref{sec:02} introduces
monomial orders and their basic properties which will be used in the
proceeding sections. We prove Theorems~\ref{1.1}, \ref{1.2} and
\ref{1.3} in Sections \ref{sec:03}, \ref{sec:04} and \ref{sec:05},
respectively. 
 
\section{Monomial Orders}
\label{sec:02}
  Let $D$ be a finite dimensional central division algebra over a
  non-Archimedean local field $k,$ and
  $\mathcal{O}_D$ be the valuation ring. Let
  $\mathfrak{P}$ be the unique maximal ideal of $\mathcal{O}_D$ and
  $\pi=\pi_D$ be a uniformizer of $D.$
  In this section, we introduce a class of
  $\mathcal{O}$-orders in $\mathrm{Mat}_n(D)$ which we call
  $\emph{monomial orders.}$
  For any integral matrix $\underline{m} = (m_{ij}) \in
   \mathrm{Mat_n(\mathbb{Z})},$ we set
  $$ \Mat_n(\mathcal{O}_D,\underline{m}) :=
  (\mathfrak{P}^{m_{ij}}) \subset \mathrm{Mat}_n(D).$$
The $\mathcal{O}_D$-submodule
$\mathrm{Mat}_n(\mathcal{O}_D,\underline{m})$ is a subring if and
only if the following conditions hold:
\begin{align*}
(\mathrm{i}&)\  m_{ii} = 0 \  \emph{for all} \  1 \leq i \leq n,\
\mathrm{and}\\
(\mathrm{ii}&)\  \mathrm{for \  all}\  1 \leq i,j,k \leq n,
\mathrm{one\ has\
the\  equality}
\end{align*}
\begin{align}
  \label{eq:21}
  m_{ik} \leq m_{ij} + m_{jk}
\end{align}
We call $\mathrm{Mat}_n(\mathcal{O}_D,\underline{m})$ the
$\emph{standard monomial order of level}\  \underline{m}.$
Let $$T := \mathrm{diag}(\mathcal{O}_D,\ldots,\mathcal{O}_D) \subset
\mathrm{Mat}_n{(\mathcal{O}_D)}$$ be the diagonal
$\mathcal{O}$-subalgebra.

\begin{lemma}\label{21} \
\begin{enumerate}
\item Any $\mathcal{O}_D$-submodule $M$ of $\mathrm{Mat}_n(D)$ which
is stable under the left and right multiplication of $T$ has the
form $\mathrm{Mat}_n(\mathcal{O}_D,\underline{m})$ for some
$\underline{m} \in \mathrm{Mat}_n(\mathbb{Z}).$
\item Any $\mathcal{O}$-order $R$ which contains a standard
monomial order $\mathrm{Mat}_n(\mathcal{O}_D,\underline{m})$ is
again a standard monomial order.
\end{enumerate}
\end{lemma}
\begin{proof}
(1) This is elementary. (2) Since
$\mathrm{Mat}_n(\mathcal{O}_D,\underline{m})$ contains $T,$ the
order $R$ is stable under the left and right multiplication of $T.$
Then the statement follows from (1). \qed
\end{proof}

 Let $\widetilde{W} := N_{\mathrm{GL}_n(D)}(T)/T^\times$ be the
 extended Weyl
 group of $\mathrm{GL}_n(D).$ We have an isomorphism $\mathbb{Z}^n
 \rtimes S_n \cong
 \widetilde{W}$ by sending $(a_1,\ldots,a_n)\cdot \sigma$ to
 diag($\pi^{a_1},\ldots,\pi^{a_n})\sigma.$ The group
 $\widetilde{W}$ acts on the set of all standard monomial orders in
 $\mathrm{Mat}_n(D)$ by conjugation. A basic question is to find
 good representatives for isomorphism classes among standard
 monomial orders.

\begin{lemma} \label{22}
Any standard monomial order
$\mathrm{Mat}_n(\mathcal{O}_D,\underline{m})$ is conjugate by an
element of $\widetilde{W}$ to an $\mathcal{O}$-order
$\mathrm{Mat}_n(\mathcal{O}_D,\underline{m'})$ such that $m'_{1j} =
0$ and $m'_{ij} \geq 0$ for $1 \leq i,j \leq n.$
\end{lemma}
\begin{proof}
Choose $a_j$ such that $a_j + m_{1j} = 0$ for $2 \leq j \leq n.$
Conjugating by the element diag($\pi^{a_1},\ldots,\pi^{a_n}),$
we get the order $\mathrm{Mat}_n(\mathcal{O}_D,\underline{m'})$ with
$m'_{1j} = 0$ for $2 \leq j \leq n.$ Using the inequality
$$m'_{1j} \leq m'_{1i} + m'_{ij}$$ we have $m'_{ij} \geq 0$
for all $1 \leq i,j \leq n.$ This proves the lemma. \qed
\end{proof}

A standard monomial order
 $\mathrm{Mat}_n(\mathcal{O}_D,\underline{m})$ is said to be of
 $\emph{positive type}$\  if $m_{ij} \geq 0$ for all $i,j$; it is
 said to be $\emph{upper
 triangular}$ if $m_{ij} = 0$ for all $i\leq j$, that is, the matrix
 $\ul m$ is strict lower triangular in the usual sense.
 By Lemma~\ref{22},
 we can reduce the situation to the case of positive type.
 In the remaining of this paper, $R$ is assumed to be a
 standard monomial order of positive type in case we do not mention.

\section{Gorenstein monomial orders}
\label{sec:03}

\subsection{Projective modules}
\label{sec:31}

Let $V=D^n$, viewed as a
right column vector space over $D$ of dimension
  $n.$ So $V$ is a left $\Mat_n(D)=\End_D(V)$-module by usual
  multiplication. Using the property that $T\subset R$,
  one shows that any non-zero left $R$-module $M \subset V$
  has
  the form $[\mathfrak{P}^{l_1},\ldots,\mathfrak{P}^{l_n}]^t$ for some
  integers $l_i$. In this case 
  $M$ is said to be
  of type $[l_1,\ldots,l_n]^t$. By the $R$-module structure of $M$,
  we have the following inequalities
  \begin{align}
  \label{eq:22}
 m_{ij} + l_{j} \geq l_i
\end{align}
for all $1 \leq i,j \leq n.$

Now we shall determine when $M$ is a projective $R$-module.
In fact we shall show that $M$ is a projective $R$-module if and
only if $M$ is, up to scalar in $D$,
a column of the monomial order $R.$
By Lemma~\ref{22}, we may assume that $m_{1j} = 0$ for
all $j = 1,\ldots,n.$ The condition (\ref{eq:22}) gives $l_1 \leq
l_i$ for all $i.$ Replacing $M$ by $\mathfrak{P}^{-l_1}M$, we may
assume that $l_1 = 0.$


\begin{thm}\label{31}
Let $R\subset \Mat_n(D)$ 
be a standard monomial order of level $\underline{m}$ with $m_{1i}=0$ for
all $i=1,\ldots,n$ and $ M =
[\mathfrak{P}^{l_1},\ldots,\mathfrak{P}^{l_n}]^t$ be an $R$-module.
Then $M$ is $R$-projective if and only if it is, up to scalar in $D$,
a column of $R$.
\end{thm}
\begin{proof}
The if part is obvious. Suppose $M$ is a projective $R$-module.
Again we may assume that $l_1=0$ and $l_i\ge 0$ for all $i$.
It follows from (\ref{eq:22}) that $l_i \leq m_{i1}$ for
$i=1,\ldots,n.$
Consider the map $\varphi: R \rightarrow M$ defined
by $\varphi(X)=X\cdot \mathbf{v},$ where $X\in R$ and
$\mathbf{v} =[\pi^{l_1},\pi^{l_2},\ldots,\pi^{l_n}]^t$.
Note that $\varphi(I_n) = \mathbf{v}$ and $\varphi$ is a surjective
$R$-homomorphism, where $I_n$ is the identity element of $R.$ Since
$M$ is $R$-projective, there exists an $R$-linear homomorphism $\psi : M
\rightarrow R$ such that $\varphi \circ \psi =id_{M}$. Clearly
$\psi(\mathbf{v}) \in I_n + \ker(\varphi).$

Let
\[ X=\left(
\begin{array}{cccc}
a_{11} & a_{12} & \ldots & a_{1n}\\
a_{21}\pi^{m_{21}} & a_{22} & \ldots & a_{2n}\pi^{m_{2n}}\\
\vdots & \vdots & \ddots & \vdots\\
a_{n1}\pi^{m_{n1}} & a_{n2}\pi^{m_{n2}} & \ldots & a_{nn}
\end{array} \right) \]
be an element in $R$, where $a_{ij}$ are elements in $\calO_D$. 
Then $X \in \ker(\varphi)$ if and only if the following system of
equations is satisfied
\begin{align}
 \label{eq:29}
  \left\{
\begin{array}{cccc}
&a_{11} + a_{12}\pi^{l_2} +\ldots+ a_{1n}\pi^{l_n} & = 0,\\
&a_{21}\pi^{m_{21}} + a_{22}\pi^{l_2} +\ldots +
a_{2n}\pi^{m_{2n}+l_n}
 & = 0,\\
&\vdots &\vdots \\
&a_{n1}\pi^{m_{n1}} + a_{n2}\pi^{m_{n2}+l_2} +\ldots + a_{nn}\pi^{l_n}
& = 0.\\
\end{array} \right.
\end{align}
Using (\ref{eq:22}) we solve the above system of
equations and get
\begin{align}
 \label{eq:210}
  \left\{
\begin{array}{llll}
a_{11} &= -a_{12}\pi^{l_2} - a_{13}\pi^{l_3} - \cdots -
a_{1n}\pi^{l_n},\\
a_{22} &= -a_{21}\pi^{m_{21}-l_2} + a_{23}\pi^{m_{23}+l_3 -l_2}-\cdots
- a_{2n}\pi^{m_{2n}+l_n -l_2},\\
&\vdots\\
 a_{nn} &= -a_{n1}\pi^{m_{n1}-l_n} -
a_{n2}\pi^{m_{n2}+l_2-l_n}-\cdots -
a_{n,n-1}\pi^{m_{n,n-1}+l_{n-1}-l_n}.
\end{array} \right.
\end{align}
Plugging (\ref{eq:210}) into $X$, we have
\begin{align}
X =\left[
\begin{array}{cccc}
-\sum\limits_{i\neq 1}a_{1i}\pi^{l_i} & a_{12} & \ldots &a_{1n} \\
a{_{21}\pi^{m_{21}}} &   -\sum\limits_{i\neq 2}a_{2i}\pi^{m_{2i} + l_i
  - l_2} & \ldots& a_{2n}\pi^{m_{2n}} \\
\vdots & \vdots & \ddots & \vdots\\
a_{n1}\pi^{m_{n1}} & a_{n2}\pi^{m_{n2}} & \ldots &
-\sum\limits_{i\neq n}a_{ni}\pi^{m_{ni} + l_i - l_n}
\end{array} \right].
\end{align}
Now, fix an element $B:=\psi(\mathbf{v})\in I_n+\ker(\varphi).$
Write
 \begin{align}
 \label{eq:211}
B =\left[
\begin{array}{cccc}
1-\sum\limits_{i\neq 1}a_{1i}\pi^{l_i} & a_{12} & \cdots &a_{1n} \\
a{_{21}\pi^{m_{21}}} &   1-\sum\limits_{i\neq 2}a_{2i}\pi^{m_{2i} +
  l_i - l_2} & \cdots& a_{2n}\pi^{m_{2n}} \\
\vdots & \vdots & \ddots & \vdots\\
a_{n1}\pi^{m_{n1}} & a_{n2}\pi^{m_{n2}} & \cdots &
1-\sum\limits_{i\neq n}a_{ni}\pi^{m_{ni} + l_i - l_n}
\end{array} \right]
\end{align}
for some elements $a_{ij}$ in $\calO_D$.  
Let $E_{ij}$ denote the  matrix in $\Mat_n(D)$ with entry
one at $(i,j)$ and zero elsewhere.
It is easy to see $\pi^{l_j} E_{11}\mathbf{v}=E_{1j}\mathbf{v}$.
Since $\psi$ is $R$-linear, we have
\begin{align}
\label{eq:212}  \pi^{l_j} E_{11}\psi(\mathbf{v})
=\psi(E_{1j}\mathbf{v}) = E_{1j}\psi(\mathbf{v}), \ \
j = 2,\ldots,n.
\end{align}
Namely, we have
\begin{align}
\label{eq:312}
 \pi^{l_j} \left[
\begin{array}{cccc}
1-\sum\limits_{i\neq 1}a_{1i}\pi^{l_i} &a_{12}& \ldots & a_{1n}\\
0&  0& \ldots & 0\\
\vdots & \vdots &  \ddots & \vdots\\
0 & 0& \ldots& 0
\end{array} \right] =\left[
\begin{array}{cccc}
a_{j1}\pi^{m_{j1}} &a_{j2}\pi^{m_{j2}}& \ldots & a_{jn}\pi^{m_{jn}}\\
0&  0& \ldots & 0\\
\vdots & \vdots &  \ddots & \vdots\\
0 & 0& \ldots& 0
\end{array} \right]
\end{align}
for $j = 2,\ldots,n$.
We shall 
conclude the theorem by comparing the valuations of entries in
both sides of (\ref{eq:312}).
Recall that $l_j \leq m_{j1}$ for $j = 1,2,\ldots, n$ from (\ref{eq:22}).
\begin{itemize}
\item[(a)] If $1-\sum\limits_{i\neq 1}a_{1i}\pi^{l_i}$ is a unit, then
\[ l_j = v(\pi^{l_j}(1-\sum\limits_{i\neq 1}a_{1i}\pi^{l_i}))=
v(a_{j1}\pi^{m_{j1}})=v(a_{j1}) + m_{j1} \geq m_{j1} \]
for $j=2,\ldots,n.$ As $l_1=m_{11}=0$, we get $l_j = m_{j1}$ 
for $j=1,\ldots,n.$ Therefore, the module $M$ is equal to  
the first column of $R.$\\
\item[(b)]If $1-\sum\limits_{i\neq 1}a_{1i}\pi^{l_i}$ is not a unit,
then there exists $k \in \{2,\dots,n\}$
such that $v(a_{1k})=l_k=0.$ We have $l_1=m_{1k}=0$ and $l_k=m_{kk}=0$.
For $j \neq 1$
and $j \neq k,$ we have
\[ l_j = v(\pi^{l_j}a_{1k})=v(a_{jk}\pi^{m_{jk}}) \geq
m_{jk}.\]
On the other hand, $m_{jk} = m_{jk} + l_k \geq l_j.$ Therefore,
$l_j = m_{jk}$ for $j=1,\ldots,n.$ In this case, the module $M$
is equal to the $k$-th column of $R.$
\end{itemize}
This completes the proof of the theorem. \qed
\end{proof}

\subsection{Gorenstein monomial orders}
\label{sec:33}
Recall that an $\mathcal{O}$-order $R$ in a semi-simple and separable
$k$-algebra is
$\emph{Gorenstein}$ if every $R$-exact sequence of right $R$-lattices
\begin{equation}
  \label{eq:3.15}
  0\rightarrow R \rightarrow M \rightarrow N \rightarrow 0
\end{equation}
splits.
Taking the $\calO$-linear dual of (\ref{eq:3.15}), we have the
following exact sequence of left $R$-lattices
\begin{equation}
  \label{eq:3.16}
  0\rightarrow \Hom_\calO(N,\calO)  \rightarrow \Hom_{\calO}(M,\calO)
  \rightarrow \Hom_\calO (R,\calO) \rightarrow 0
\end{equation}
as $N$ is $\calO$-projective. The exact sequence (\ref{eq:3.15})
splits if and only if the exact sequence (\ref{eq:3.16}) splits as
$R$-modules. Therefore $R$ is Gorenstein if and only if the
$\calO$-linear dual $R^\vee:=\Hom_\calO(R,\calO)$ is a left projective
$R$-module; cf. \cite[Prop. 6.1, p. 1363]{DKR:1967}.

Recall that an $\mathcal{O}$-order $R$ in a semi-simple and separable
$k$-algebra is
$\emph{Bass}$ if any $\calO$-order containing it is Gorenstein.

\begin{lemma}\label{32} Let $R=\Mat_n(\calO_D,\ul m)$ be a monomial
  order of level $\ul m=(m_{ij})$ in $\Mat_n(D)$. 
  As left $R$-modules, $R^\vee$ is isomorphic to $\Mat_n(\calO_D,\ul
  m')\subset \Mat_n(D)$ with $m'_{ij}=-m_{ji}$ for all $1\le i,j\le
  n$.
\end{lemma}
\begin{proof}
  View $A:=\Mat_n(D)$ as a right $A$-module by usual
  multiplication.
 Consider the $k$-bilinear
 pairing $A\times A \to k$ by
$(x,y)\mapsto \tr_{A/k}(xy)$, where $\tr_{A/k}$ denotes the reduced
trace from $A$ to $k$. For any element $x\in A$, we denote by
$f_x\in A^\vee:=\Hom_k(A_A,k)$ the map $f_x(y):=\tr_{A/k}(yx)$. The
map is left $A$-linear as
\[ a\cdot f_x(y)=f_x(ya)=\tr_{A/k}(yax)=f_{ax}(y)\quad \forall\,
a,x,y\in A. \]
Fix this isomorphism $A\simeq A^\vee$, we have
\[ R^\vee\simeq R^*:=\{\, x\in A\, ;\, \tr_{A/k}(ax)\in \calO, \  \forall\, a\in
R\,\}. \]
Let $\calD^{-1}:=\{x\in D; \tr_{D/k}(\calO_D x)\subset \calO\}$ be the
``inverse of different'' of $D/k$. As $R=\Mat_n(\calO_{D},\ul m)$ with
level $\ul m =(m_{ij})$. Then one
computes (using the dual basis of $\pi^{m_{ij}} E_{ij}$) that
\[ R^*=( \grP^{m_{ij}'}\calD^{-1})_{i,j}, \]
where $m_{ij}'=-m_{ji}$ for all $i,j$. As a left $R$-module, $R^\vee$
is isomorphic to $\Mat_n(\calO_D,\ul m')$. This proves the lemma. \qed
\end{proof}



By Theorem~\ref{31} and Lemma~\ref{32}, we have proven
the following result. This will be used in the next section
to characterize Gorenstein orders among upper triangular monomial
orders (Theorem~\ref{42}).

\begin{thm} \label{33}
 Let $R=\Mat_n(\calO_D, \ul m)$ be a standard monomial order
  of level $\ul
  m=(m_{ij})$ in the central simple algebra $\Mat_n(D)$. Then $R$ is
  Gorenstein if and only if for each $1\le i\le n$ there exists an
  integer $c(i)$ such that the integral column vector
  $[-m_{i1}+c(i),\dots ,-m_{in}+c(i)]^t$ is equal to a column of $\ul m$.
\end{thm}

\subsection{}
\label{sec:34}
Using Theorem~\ref{33}, we give a list of all Gorenstein monomial
orders $R$ in $\Mat_n(D)$ with $n=4$ (when $n=3$ any Gorenstein
monomial order is isomorphic to an upper triangular monomial order).
Any such order $R$ is isomorphic to $\Mat_4(\calO_D, \ul m)$, where
$\ul m$ is one of the following:
\[
\left[
\begin{array}{cccc}
0 & 0 & 0 & 0 \\
0 & 0 & 0 & 0 \\
0 & 0 & 0 & 0 \\
0 & 0 & 0 & 0
\end{array} \right],
\left[
\begin{array}{cccc}
0 & 0 & 0 & 0 \\
0 & 0 & 0 & 0 \\
0 & 0 & 0 & 0 \\
a & a & a & 0
\end{array} \right],
\left[
\begin{array}{cccc}
0 & 0 & 0 & 0 \\
0 & 0 & 0 & 0 \\
a & a & 0 & 0 \\
a & a & 0 & 0
\end{array} \right],
\left[
\begin{array}{cccc}
0 & 0 & 0 & 0 \\
0 & 0 & 0 & 0 \\
a & a & 0 & 0 \\
a & a & a & 0
\end{array} \right], \]
\[
\left[
\begin{array}{cccc}
0 & 0 & 0 & 0 \\
a & 0 & 0 & 0 \\
a & a & 0 & 0 \\
a & a & a & 0
\end{array} \right],
\left[
\begin{array}{cccc}
0 & 0 & 0 & 0 \\
a & 0 & a & 0 \\
b & b & 0 & 0 \\
a+b & b & a & 0
\end{array} \right],
\left[
\begin{array}{cccc}
0 & 0 & 0 & 0 \\
a & 0 & a & 0 \\
a+b & b & 0 & 0 \\
a+b & a+b & a & 0
\end{array} \right],
\]
for some positive integers $a, b$.

\section{Upper triangular Gorenstein orders}\label{sec:04}

We call an $\calO$-order $R$ in a central simple algebra $A\simeq
\Mat_n(D)$ over $k$ an {\it Eichler order} if it is isomorphic to an
upper triangular monomial order $\Mat_n(\calO_D,\ul m)$ of level $\ul
 m=(m_{ij})$ such that for any $n\ge i>j\ge 1$, either $m_{ij}=0$ or
 $m_{ij}=m_{n1}$. In this case, there are a tuple $(k_1,\dots, k_t)\in
 \Z_{>0}^t$ with $k_1+\dots+k_t=n$ and a positive integer $a$ if
 $t>1$ so that $\ul m$ can be written as a $t$ by $t$ blocks
 $(M_{kl})_{1\le k,l\le t}$ of matrices with
\[ M_{kl}=
\begin{cases}
  {\bf 0}^{n_k}_{n_l} & \text{if $k\le l$;}\\
  a{\bf 1}^{n_k}_{n_l} & \text{if $k>l$.}
\end{cases} \]
Recall that ${\bf 1}^{r}_{s}$ (resp. ${\bf 0}^{r}_{s}$) denotes the
$r\times s$ matrix with every entry equal to $1$ (resp. $0$).
We shall call $t$ the {\it period} of the Eichler order $R$ and
  $(k_1,\dots, k_t)$ is the {\it invariant} of $R$, which is uniquely
  determined by $R$ up to cyclic permutation. For example if $t=3$,
  $(k_1,k_2,k_3)=(1,2,1)$ and $a=2$, then
\[ \ul m=\left[
\begin{array}{cccc}
0 & 0 & 0 & 0 \\
2 & 0 & 0 & 0 \\
2 & 0 & 0 & 0 \\
2 & 2 & 2 & 0
\end{array} \right].
\]

It is not hard to show that any Eichler order is a Gorenstein order using
Lemma~\ref{32} and Theorem~\ref{33}.
The main result (Theorem~\ref{42}) says that the converse is also
true. That is, any Gorenstein upper triangular monomial order is an
Eichler order.






\begin{thm}\label{42}
If $R=\Mat_n(\calO_D,\ul m)$ is an upper triangular monomial order,
then $R$ is Gorenstein if and only if $R$ is an Eichler order.
\end{thm}
\begin{proof}
The if part is easier; this follows from a direct computation of
$R^\vee$ and Theorem~\ref{33} for $R$-projectivity  of $R^\vee$. 
We leave the detailed proof to the reader and prove
the other direction.
Suppose $R$ is a Gorenstein order and write
\begin{align*} \underline{m} = \left[
\begin{array}{ccccccc}
0& 0&0  & 0& \ldots & 0\\
m_{21} & 0 & 0 &  0 & \ldots &0 \\
m_{31} & m_{32} & 0 & 0& \ldots &0 \\
m_{41} & m_{42} & m_{43} & 0& \ldots &0 \\
\vdots & \vdots & \vdots & \vdots & \ddots & \vdots\\
m_{n1} & m_{n2}& m_{n3} & m_{n4} &\ldots & 0
\end{array} \right].
\end{align*}
We must prove that $R$ is an Eichler order, or
equivalently $m_{jk}\in \{m_{n1},0\}$ for all $1\le k<j \le n$.

By Lemma~\ref{32}, the dual $R^{\vee}$ is
is isomorphic to the module $R' =
\mathrm{Mat}_n(\mathcal{O}_D,\underline{m}')$ with level
\begin{equation}
  \label{eq:41}
 \underline{m}' = \left[
\begin{array}{cccccc}
0& 0 & 0 & 0 & \ldots & 0\\
0 &m_{21} & m_{31}-m_{32} & m_{41} - m_{42} & \ldots & m_{n1} - m_{n2}\\
0 &m_{21} & m_{31} & m_{41}- m_{43} &\ldots & m_{n1} - m_{n3}\\
0 &m_{21} & m_{31} & m_{41} &\ldots & m_{n1} - m_{n4}\\
\vdots & \vdots & \vdots & \vdots & \ddots & \vdots\\
0 & m_{21}& m_{31} & m_{41} & \ldots & m_{n1}
\end{array} \right].
\end{equation}
Here we normalize $\ul m'$ so that its first row is zero.
Since $R$ is Gorenstein, by Theorem~\ref{33}, 
any column of $\underline{m}'$ is a column of $\underline{m}$.

We prove the statement that $R$ is Eichler by induction on $n$.
The cases $n=1,2$ are obvious.
Suppose $n\ge 3$ and that the statement is true for $n'< n$. 

(a) Suppose $m_{21} = 0$. It follows from $m_{i1} \geq m_{i2} = m_{i2} +
m_{21} \geq m_{i1}$ that $m_{i2} = m_{i1}$ for all $i =
1,\ldots,n$. Namely,
\[
\underline{m} =\left[
\begin{array}{lc}
0 & {\bf 0}^{1}_{n-1} \\
\mathbf{Y} &\underline{m_1}
\end{array} \right] \ \   \mathrm{and}\ \
 \underline{m}'= \left[
\begin{array}{ll}
0 & {\bf 0}^{1}_{n-1}\\
{\bf 0}^{n-1}_{1} &\underline{m_1'}
\end{array} \right],
\]
 where
\[
\underline{m_1} = \left[
 \begin{array}{ccccc}
 0 &0 & 0 &\ldots &0\\
 m_{31}&0 &0 & \ldots &0\\
 m_{41}& m_{43}& 0 & \ldots &0\\
  \vdots & \vdots & \vdots & \ddots & \vdots\\
  m_{n1}& m_{n3} &m_{n4}& \ldots &0
 \end{array} \right], \ \
\mathbf{Y} = \left[
 \begin{array}{ccccc}
0\\
m_{31} \\
m_{41}\\
\vdots \\
m_{n1}
 \end{array} \right],
\]
and
\[
 \underline{m_1'} = \left[
 \begin{array}{ccccc}
 0  & 0 &0& \ldots & 0\\
 0  & m_{31} &m_{41} - m_{43} & \ldots & m_{n1} - m_{n3}\\
 0  & m_{31} & m_{41} &\ldots & m_{n1} - m_{n4}\\
 \vdots & \vdots & \vdots & \ddots & \vdots\\
 0 & m_{31} &m_{41} & \ldots & m_{n1}
 \end{array} \right].
\]
Put $R_1:= \mathrm{Mat}_{n-1}(\mathcal{O}_D,\underline{m_1}).$ Then
the dual lattice $R_1^{\vee}$ is isomorphic to the module $R_1' :=
\mathrm{Mat}_{n-1}(\mathcal{O}_D,\underline{m_1'}).$ Clearly $R$ is
an Eichler order if and only if so is $R_1$. By the induction
hypothesis, $R_1$ is Eichler. It follows that $R$ is
an Eichler order.


(b) Suppose $m_{21} \neq 0.$ Then the second column of the matrix
$\underline{m}'$ must be the first column of $\underline{m}$ and
hence $m_{21} = m_{31} = \ldots = m_{n1}.$ So the matrix $\ul m'$
becomes
\[
\underline{m}' = \left[
\begin{array}{cccccc}
0& 0 & 0 & 0 & \ldots & 0\\
0 &m_{n1} & m_{n1}-m_{32} & m_{n1} - m_{42} & \ldots & m_{n1} - m_{n2}\\
0 &m_{n1} & m_{n1} & m_{n1}- m_{43} &\ldots & m_{n1} - m_{n3}\\
0 &m_{n1} & m_{n1} & m_{n1} &\ldots & m_{n1} - m_{n4}\\
\vdots & \vdots & \vdots & \vdots & \ddots & \vdots\\
0 & m_{n1}& m_{n1} & m_{n1} & \ldots & m_{n1}
\end{array} \right].
\]
As each entry of the $k$-th row of $\underline{m}'$, for $2\leq k \leq
n-1$, occurs in the $k$-th row of $\ul m$, we
have
\begin{align}\label{eq:415}
\displaystyle m_{n1} - m_{jk} \in \{m_{k1}, m_{k2},\ldots, \
m_{k,k-1},0 \}\ \  \mathrm{for \ all}\ 2\le  k < j \leq n.
\end{align}

We now prove
\begin{align}\label{eq:42}
 \{m_{k1}, m_{k2},\ldots, \
m_{k,k-1},0 \} = \{m_{n1},0\}.
\end{align}
 for $2 \leq k \leq
n-1$, or equivalently $m_{jk} \in \{ 0,m_{n1} \}$ for all $2\le k < j \leq
n$ (as  $m_{n1}-m_{jk} \in \{ m_{n1}, 0 \}$ by (\ref{eq:415})). 
We prove this by induction on $k$. When $k = 2$, (\ref{eq:42}) is true
as $m_{21}=m_{n1}$, and by (\ref{eq:415}) one has  $m_{j2} \in \{0,
m_{n1}, 0\}$ for $2 < j \leq n.$   
Suppose the statement (\ref{eq:42}) is true for $k \leq l$, i.e.
$m_{jk} \in \{ 0, m_{n1}\}$ for all $k < j \leq n$ and $k =
2,\ldots,l$. Then for $k = l + 1,$ we have $$m_{n1} - m_{j,l+1} \in
\{m_{l+1,1},m_{l+1,2},\ldots,m_{l+1,l}, 0\} = \{m_{n1},0\},$$ for all
$l+1 < j$ because of $m_{l+1,1} = m_{n1}$ and the induction
hypothesis. This proves (\ref{eq:42}) and that $R$ is Eichler in case
(b). 

This completes the proof of the theorem. \qed
\end{proof}



\section{Monomial orders and Bass orders}
\label{sec:05}

\subsection{}
\label{sec:51}
In this section we prove Theorem~\ref{1.3}.


 \begin{prop}\label{51}
  Assume that  $R = \Mat_n(\calO_{D},\ul m)$ is upper triangular.
  Then $R$ is
  Bass if and only if either $R$ is a hereditary order, or $R$ is an
  Eichler order of period two.
 \end{prop}
\begin{proof}
We prove the if part. This is obvious if $R$ is hereditary.
Suppose $R$ is an Eichler order of
period two. We may assume $R =
\Mat_n(\calO_D, \ul
  m')$ with
\[ \ul m'= \left[
\begin{array}{cc}
{\bf 0}^{k_1}_{k_1} &  {\bf 0}^{k_1}_{k_2} \\
a\, {\bf 1}^{k_2}_{k_1} & {\bf 0}^{k_2}_{k_2}
\end{array} \right],\]
where $k_1,k_2, a$ are positive integers with $k_1+k_2=n$.
Clearly $R$ is Gorenstein and any overring $S\supset R$
is $\Mat_n(\calO_D, \underline{w})$ with
\[ \underline{w}= \left[
\begin{array}{cc}
{\bf 0}^{k_1}_{k_1} & {\bf 0}^{k_1}_{k_2} \\
b\, {\bf 1}^{k_2}_{k_1} & {\bf 0}^{k_2}_{k_2}
\end{array} \right] \]
for some non-negative integer $b\leq a$, which is also
Gorenstein. This shows that $R$ is Bass.


 Conversely, suppose $R$ is Bass.  As $R$ is Gorenstein,  by
 Theorem~\ref{42}, $R$ is isomorphic to an Eichler
 order of period $t.$ If $t=1$, then $R$ is a maximal order.
 Assume $t>1$.
 Without lose of generality,
 we may assume $R = \Mat_n(\calO_D, \underline{m}')$ where
\[ \ul m'= \left[
\begin{array}{cccc}
 {\bf 0}^{k_1}_{k_1} & {\bf 0}^{k_1}_{k_2} & \cdots &  {\bf
   0}^{k_1}_{k_t}  \\
a\, {\bf 1}^{k_2}_{k_1} &  {\bf 0}^{k_2}_{k_2} &  \cdots & {\bf
   0}^{k_2}_{k_t}  \\
\vdots &  \vdots & \ddots & \vdots \\
a\, {\bf 1}^{k_t}_{k_1} & a\,  {\bf 1}^{k_t}_{k_2} &  \cdots &
{\bf 0}^{k_t}_{k_t}
\end{array} \right],
\]
 for some $a \in \Z_{>0}.$ If $a = 1,$ then $R$ is hereditary,
 which is a Bass order. The case $t=2$ is also possible.
 We need to show that $R$ is not Bass if $a > 1$ and $t \geq 3$. In
 this case, take $S = \Mat_n(\calO_D,
 \underline{w})$ with
 \[ \ul w= \left[
\begin{array}{cccc}
 {\bf 0}^{k_1}_{k_1} &  {\bf 0}^{k_1}_{k_2} & \cdots &
  {\bf 0}^{k_1}_{k_t}  \\
 {\bf 1}^{k_2}_{k_1} &   {\bf 0}^{k_2}_{k_2} &  \cdots &
  {\bf 0}^{k_2}_{k_t}  \\
\vdots &  \vdots & \ddots & \vdots \\
 {\bf 1}^{k_{t-1}}_{k_1} &  {\bf 1}^{k_{t-1}}_{k_2} &  \cdots
&  {\bf 0}^{k_{t-1}}_{k_t} \\
2\, {\bf 1}^{k_t}_{k_1} &  {\bf 1}^{k_t}_{k_2} &  \cdots
&  {\bf 0}^{k_t}_{k_t}
\end{array} \right].
\] Then $S \supseteq R$ and $S$ is an $\mathcal{O}$-order. But $S$ is
  not
Gorenstein, by Theorem \ref{42}. This proves the proposition. \qed
\end{proof}

\begin{lemma}\label{52}
 Let $R = \Mat_n(\calO_{D},\ul m)$ be a Bass order with level
 $\underline{m}.$ Take $R' = \Mat_n(\calO_{D},\ul m')$, where
 $\underline{m}' = (m'_{ij})$ is defined by
 $m'_{ij} = 1$
 if $m_{ij} \neq 0$ and $m'_{ij} = 0 $ if $m_{ij} = 0.$ Then $R'$ is a
 Bass order.
\end{lemma}
\begin{proof}
It suffices to check that $R'$ is an order, that is, the inequalities
 $m'_{ij} + m'_{jk} \geq m'_{ik}$ hold
for all $ 1 \leq i,j,k \leq n.$ If $m'_{ik} = 0,$ then this is
obvious. Suppose $m'_{ik} = 1.$ Since $ 0 < m_{ik} \leq m_{ij} +
m_{jk},$ we have $m_{ij} > 0$ or $m_{jk} > 0$, which implies $m'_{ij}
= 1$ or $m'_{jk} = 1.$ Therefore, we have $1 = m'_{ik} \leq m'_{ij} +
m'_{jk}.$ \qed
\end{proof}

\begin{prop}\label{53}
Let $R = \Mat_n(\calO_{D},\ul m)$ be a monomial order of level
$\underline{m} = (m_{ij})$ with each entry $m_{ij} \in
\{0,1\}.$ If $R$ is Gorenstein, then $R$ is, up to conjugation, of
upper triangular type.
 \end{prop}
\begin{proof}
By Lemma~\ref{22}, we can assume $m_{1i} = 0$ for all $i =
1,\ldots, n.$ Since the first row of $\underline{m}$ is zero and $R$
is Gorenstein, the $j$-th
column of $\ul m$ is zero for some $1 \leq j \leq n$.
Up to conjugation, we may assume $j = n$. So
we can assume
\[ \ul m= \left[
\begin{array}{cccccc}
0 & 0 & 0& \cdots & 0 & 0 \\
m_{21} & 0 & m_{23} & \cdots & m_{2,n-1} & 0 \\
m_{31} & m_{32} & 0 & \cdots & m_{3,n-1} & 0\\
\vdots & \vdots &\vdots &  \ddots & \vdots & 0 \\
m_{n1} & m_{n2} & m_{n3 } &\cdots & m_{n,n-1} & 0
\end{array} \right].
\]
We now prove the statement by induction on $n$.
When $n\le 3$, the order $R$ is upper triangular.
Hence we may assume that $n>3$.

For $1 \leq i, j \leq n$, one has $m_{ij} \leq m_{in} + m_{nj} = m_{nj}.$
If $m_{nj} = 0$ for some $1 \leq j \leq n,$ then the $j$-th column
of $\underline{m}$ is zero. After a suitable conjugation, we can
further assume
\[ \ul m= \left[
\begin{array}{ccccccc}
0 & 0 & \cdots & 0 & 0 & \cdots & 0 \\
m_{21} & 0  & \cdots & m_{2k} & 0 & \cdots & 0 \\
\vdots & \vdots & \ddots & \vdots & \vdots & \ddots & \vdots\\
m_{k1} & m_{k2} & \cdots & 0 & 0 & \cdots & 0 \\
m_{k+1,1} & m_{k+1,2} & \cdots & m_{k+1,k} & 0 & \cdots & 0\\
\vdots & \vdots &  \ddots & \vdots & \vdots & \ddots & \vdots \\
1 & 1 & \cdots & 1 & 0 & \cdots & 0
\end{array} \right],
\]
for some $1 \leq k \leq n-1.$ Note that the $n$-th row of
$\underline{m}$ is
$ [{1}^{(k)},  {0}^{(n-k)} ]$ ($1$ repeated $k$ times
and $0$ repeated $(n-k)$ times).
The dual  $R^{\vee}$ is isomorphic to $R'=\Mat_n(\calO_D, \ul m')$
with $\ul m'$ as (\ref{eq:41}).
The $n$-th
column of $\ul m'$ is $[0^{(k)}, {1}^{(n-k)} ]^t $.
Therefore, for some $1
\leq j \leq k$,
the $j$-th column of $\underline{m}$ is $[0^{(k)}, {1}^{(n-k)} ]^t $.
Up to
conjugation, we can assume $j = k.$ i.e.
\[ \ul m= \left[
\begin{array}{cccccccc}
0 & 0 & \cdots & 0&0 & 0 & \cdots & 0 \\
m_{21} & 0  & \cdots & m_{2,k-1}& 0 & 0 & \cdots & 0 \\
\vdots & \vdots & \ddots & \vdots & \vdots & \vdots & \ddots & \vdots\\
m_{k1} & m_{k2} & \cdots &m_{k,k-1} & 0 & 0 & \cdots & 0 \\
m_{k+1,1} & m_{k+1,2} & \cdots &m_{k+1,k-1} & 1 & 0 & \cdots & 0\\
\vdots & \vdots &  \ddots &\vdots & \vdots & \vdots & \ddots & \vdots \\
1 & 1 & \cdots & 1 & 1 & 0 & \cdots & 0
\end{array} \right].
\]
For each $1 \leq i \leq n-k$ and $1 \leq j \leq k,$ we have $1 =
m_{k+i,k} \leq m_{k+i,j} + m_{jk} = m_{k+i,j}$, which implies
$m_{k+i,j} =1.$ This shows
\[ \ul m= \left[
\begin{array}{cc}
\ul m_1 & {\bf 0}^k_{n-k} \\
{\bf 1}^{n-k}_k & {\bf 0}^{n-k}_{n-k} \\
\end{array} \right],  \]
where
\[ \ul m_1= \left[
\begin{array}{ccccc}
0 & 0 & \cdots & 0 & 0\\
m_{21} & 0  & \cdots & m_{2,k-1} & 0 \\
\vdots & \vdots & \ddots & \vdots & \vdots\\
m_{k1} & m_{k2} & \cdots & m_{k,k-1} & 0
\end{array} \right]. \]

Put $R_1 := \Mat_k(\calO_{D},\ul m_1)$.
Then $R$ is Gorenstein if and only if $R_1$ is
Gorenstein. By the induction hypothesis, $\ul m_1$ is, up to
conjugate,  strict lower
triangular, and so
is $\ul m$. This proves the proposition. \qed
\end{proof}

\begin{thm}\label{54}
  Let $R=\Mat_n(\calO_D, \ul m)$ be a monomial order with level $\ul
  m=(m_{ij})$ in the central simple algebra $\Mat_n(D)$. Then $R$ is
  Bass if and only if either $R$ is a hereditary order, or $R$ is an
  Eichler order of period two.

\end{thm}
\begin{proof}
If $R$ is a hereditary order or an Eichler order of period 2, then
$R$ is Bass, by Proposition~\ref{51}.
Suppose $R$ is a Bass order.
Take $R'$ as in Lemma~\ref{52}; it is a Bass order.
Proposition~\ref{53} shows that the matrix $\underline{m}'$
can be conjugated to a strict lower triangular one, 
and hence so is $\underline{m}.$
Therefore, after a suitable conjugation, we can assume that $R$ is
upper triangular.
 By Proposition~\ref{51}
 again, $R$ is either a hereditary order or an Eichler order of
 period 2. This completes the proof of the theorem. \qed
\end{proof}

\subsection{An application}
\label{sec:52}
We give an example which shows
that Bass orders are not stable under etale base change (but it is so
for quaternion algebras; 
see~\cite[p. 507]{brzezinski:order83}).
Let $D$ be the division quaternion algebra over a non-Archimedean local
field $k$. Put
\[ R= \left[
\begin{array}{cc}
\calO_D  &  \calO_D \\
 \pi_D^2 \calO_D & \calO_D
\end{array} \right],\]
which is a Bass order.
Let $k_2$ be the unique unramified field extension of $k$ and let
$\calO_2$ be the ring of integers. Then
\[ R\otimes_\calO \calO_2 \simeq \Mat_4(\calO_2, \ul m) \]
with
\[ \ul m=\left[
\begin{array}{cccc}
 0 & 0  & 0 & 0  \\
 1 & 0  & 1 & 0  \\
 1 & 1  & 0 & 0  \\
 2 & 1  & 1 & 0  \\
\end{array} \right].\]
The order $\Mat_4(\calO_2, \ul m)$ is not a Bass order because it is
contained in the order $\Mat_4(\calO_2, \ul {\tilde m})$ with
\[  {\tilde {\ul m}}=\left[
\begin{array}{cccc}
 0 & 0  & 0 & 0  \\
 1 & 0  & 0 & 0  \\
 1 & 1  & 0 & 0  \\
 2 & 1  & 1 & 0  \\
\end{array} \right],\]
which is not Gorenstein, by Theorem~\ref{42}.

\section*{Acknowledgments}

The second named author is grateful to 
J.-K. Yu for insightful discussions which
lead the present work. The authors were partially supported 
by the grants NSC 100-2628-M-001-006-MY4 and AS-98-CDA-M01.
We thank John S. Kauta for pointing out the typos 
in the proof of Theorem 3.1.

\end{document}